\documentclass[a4paper,11pt]{amsart}
\usepackage[utf8x]{inputenc}
\usepackage{amsmath,amssymb,amsthm}
\usepackage{graphics,color, graphicx}
\usepackage{colonequals}
\usepackage{textcomp}
\usepackage{url,hyperref}
\graphicspath{{Fig/}}

\newtheorem{definition}{Definition}
\newtheorem{theorem}[definition]{Theorem}
\newtheorem{prp}[definition]{Proposition}
\newtheorem{cor}[definition]{Corollary}

\theoremstyle{remark}

\newcommand{\CR}{\operatorname{cr}}

\def\Sph{\mathbb S}
\def\dist{\operatorname{dist}}

\title[Butterfly theorems and hyperbolic geometry]{A porism for cyclic quadrilaterals, butterfly theorems, and hyperbolic geometry}
\author{Ivan Izmestiev}
\thanks{Supported by the European Research Council under the European Union's Seventh Framework Programme (FP7/2007-2013)/\allowbreak ERC Grant agreement no.~247029-SDModels}
\address{Institut f\"ur Mathematik \\
Freie Universit\"at Berlin \\
Arnimallee 2 \\
D-14195 Berlin \\
 GERMANY}
\email{izmestiev@math.fu-berlin.de}

\begin{document}

\maketitle

\begin{abstract}
If there exists a cyclic quadrilateral whose sides go through the given four collinear points, then there are infinitely many such quadrilaterals inscribed in the same circle. We give two proofs of this porism; one based on cross-ratios, the other on compositions of hyperbolic isometries.
\end{abstract}

\section{The porism}
In \cite{Kocik13}, the following theorem was proved.
\begin{theorem}
\label{thm:Porism}
Let $C$ be a circle in the plane, and let $p, q, r, s$ be four collinear points not on $C$. Choose a point $x \in C$ not collinear with $p, q, r, s$ and draw a chain of four chords starting at $x$ which go consecutively through $p, q, r, s$.

If the chain closes for some choice of $x$, then it closes for any other choice.
\end{theorem}

\begin{figure}[ht]
\begin{center}
\begin{picture}(0,0)%
\includegraphics{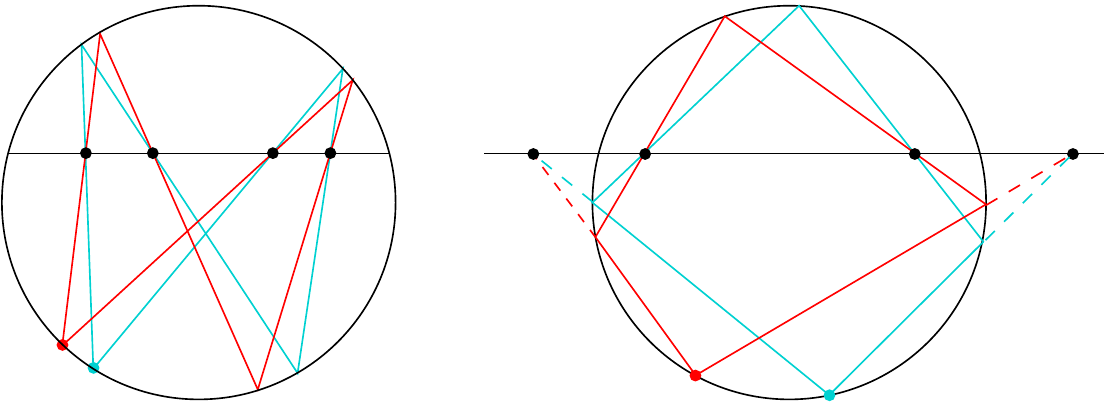}%
\end{picture}%
\setlength{\unitlength}{4144sp}%
\begingroup\makeatletter\ifx\SetFigFont\undefined%
\gdef\SetFigFont#1#2#3#4#5{%
  \reset@font\fontsize{#1}{#2pt}%
  \fontfamily{#3}\fontseries{#4}\fontshape{#5}%
  \selectfont}%
\fi\endgroup%
\begin{picture}(5060,1859)(443,-1458)
\put(568,-1243){\makebox(0,0)[lb]{\smash{{\SetFigFont{9}{10.8}{\rmdefault}{\mddefault}{\updefault}{\color[rgb]{0,0,0}$x$}%
}}}}
\put(1954,-391){\makebox(0,0)[lb]{\smash{{\SetFigFont{9}{10.8}{\rmdefault}{\mddefault}{\updefault}{\color[rgb]{0,0,0}$r$}%
}}}}
\put(1660,-394){\makebox(0,0)[lb]{\smash{{\SetFigFont{9}{10.8}{\rmdefault}{\mddefault}{\updefault}{\color[rgb]{0,0,0}$s$}%
}}}}
\put(709,-397){\makebox(0,0)[lb]{\smash{{\SetFigFont{9}{10.8}{\rmdefault}{\mddefault}{\updefault}{\color[rgb]{0,0,0}$p$}%
}}}}
\put(1027,-392){\makebox(0,0)[lb]{\smash{{\SetFigFont{9}{10.8}{\rmdefault}{\mddefault}{\updefault}{\color[rgb]{0,0,0}$q$}%
}}}}
\put(3485,-1403){\makebox(0,0)[lb]{\smash{{\SetFigFont{9}{10.8}{\rmdefault}{\mddefault}{\updefault}{\color[rgb]{0,0,0}$x$}%
}}}}
\put(2759,-397){\makebox(0,0)[lb]{\smash{{\SetFigFont{9}{10.8}{\rmdefault}{\mddefault}{\updefault}{\color[rgb]{0,0,0}$p$}%
}}}}
\put(3424,-392){\makebox(0,0)[lb]{\smash{{\SetFigFont{9}{10.8}{\rmdefault}{\mddefault}{\updefault}{\color[rgb]{0,0,0}$q$}%
}}}}
\put(4518,-391){\makebox(0,0)[lb]{\smash{{\SetFigFont{9}{10.8}{\rmdefault}{\mddefault}{\updefault}{\color[rgb]{0,0,0}$r$}%
}}}}
\put(5362,-394){\makebox(0,0)[lb]{\smash{{\SetFigFont{9}{10.8}{\rmdefault}{\mddefault}{\updefault}{\color[rgb]{0,0,0}$s$}%
}}}}
\end{picture}%

\end{center}
\caption{A porism for cyclic quadrilaterals.}
\label{fig:Porism}
\end{figure}
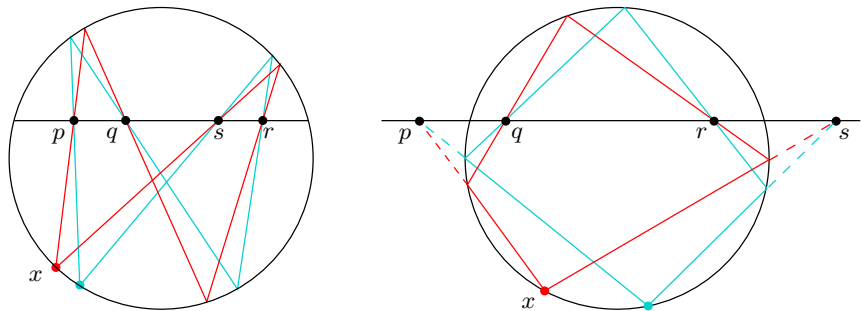

In Section \ref{sec:CR}, we give a short proof of Theorem \ref{thm:Porism} using cross-ratios and establish a link with the butterfly theorem and its projective generalization. Section \ref{sec:HypMoeb} interprets Theorem \ref{thm:Porism} in terms of hyperbolic and M\"obius geometry, reproves and generalizes it.

Both approaches to Theorem \ref{thm:Porism} are quite common and belong to the folklore in the mathematical olympiads community. We believe that they can serve as nice exercises in projective, or respectively hyperbolic, geometry.

The author thanks Arseniy Akopyan for useful remarks.

\section{Cross-ratios}
\label{sec:CR}
\subsection{The projective butterfly theorem}
Theorem \ref{thm:Porism} deals with incidencies of points and lines, thus belongs to projective geometry. Our first proof uses the basic invariant of projective transformations, the cross-ratio.

\begin{definition}
The \emph{cross-ratio} of four distinct real numbers $a$, $b$, $c$, $d$ is
\[
\CR(a,b;c,d) := \frac{a-c}{b-c} : \frac{a-d}{b-d}
\]
The cross-ratio of four collinear points is defined as the cross-ratio of their coordinates in some (and then any) affine coordinate system on the line.
\end{definition}

\begin{theorem}[The projective butterfly theorem]
\label{thm:CR}
Let $p$, $q$, $r$, $s$ be the intersection points of a line $\ell$ with the (extensions of the) consecutive sides of a cyclic quadrilateral.

% Given a quadrilateral inscribed in a circle $C$ and a line $\ell$ passing through none of its vertices, denote by $p$, $q$, $r$, $s$ the intersection points of $\ell$ with the consecutive sides of the quadrilateral.

If $\ell$ intersects the circumcircle $C$ in two different points $a$ and $b$, then
\begin{equation}
\label{eqn:CRCond}
\CR(a,b;p,q) = \CR(a,b;s,r).
\end{equation}

If $\ell$ is tangent to $C$ at a point $a$, then
\begin{equation}
\label{eqn:Tangent}
\frac{1}{a-p} - \frac{1}{a-q} = \frac{1}{a-s} - \frac{1}{a-r}
\end{equation}
where $a-p$ denotes the signed length of the segment $ap$.

If $\ell$ and $C$ are disjoint, then
\begin{equation}
\label{eqn:ImagCR}
\angle paq = \angle sar
\end{equation}
where $a$ is the point whose position is shown on Figure \ref{fig:CR}, right.

Conversely, if the circle $C$ and the points $p,q,r,s \in \ell$ satisfy the conditions \eqref{eqn:CRCond}, \eqref{eqn:Tangent}, or \eqref{eqn:ImagCR}, then for every $x \in C$ there is an inscribed quadrilateral $xyzt$ that intersects $\ell$ consecutively in the given points.
\end{theorem}
\begin{figure}[ht]
\begin{center}
\begin{picture}(0,0)%
\includegraphics{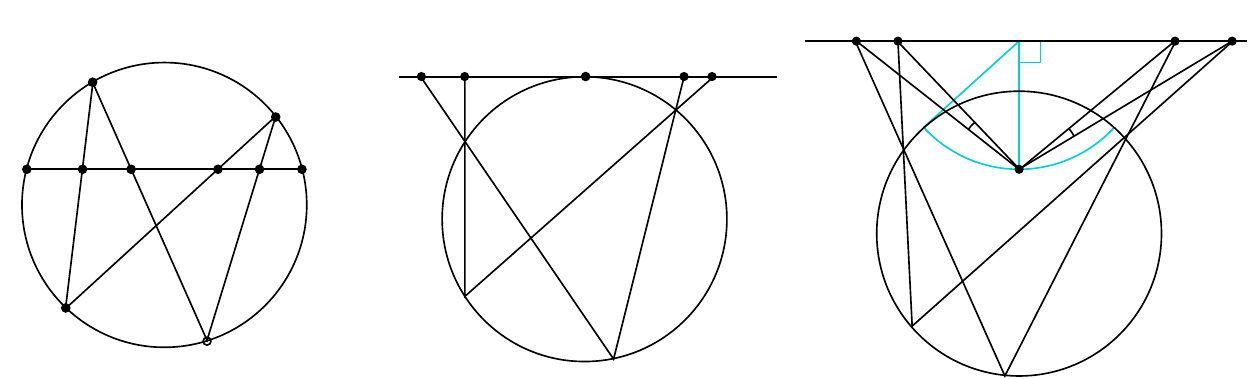}%
\end{picture}%
\setlength{\unitlength}{4144sp}%
\begingroup\makeatletter\ifx\SetFigFont\undefined%
\gdef\SetFigFont#1#2#3#4#5{%
  \reset@font\fontsize{#1}{#2pt}%
  \fontfamily{#3}\fontseries{#4}\fontshape{#5}%
  \selectfont}%
\fi\endgroup%
\begin{picture}(5714,1726)(326,-1019)
\put(850,-160){\makebox(0,0)[lb]{\smash{{\SetFigFont{9}{10.8}{\rmdefault}{\mddefault}{\updefault}{\color[rgb]{0,0,0}$q$}%
}}}}
\put(1610,190){\makebox(0,0)[lb]{\smash{{\SetFigFont{9}{10.8}{\rmdefault}{\mddefault}{\updefault}{\color[rgb]{0,0,0}$t$}%
}}}}
\put(589,-160){\makebox(0,0)[lb]{\smash{{\SetFigFont{9}{10.8}{\rmdefault}{\mddefault}{\updefault}{\color[rgb]{0,0,0}$p$}%
}}}}
\put(341,-97){\makebox(0,0)[lb]{\smash{{\SetFigFont{9}{10.8}{\rmdefault}{\mddefault}{\updefault}{\color[rgb]{0,0,0}$a$}%
}}}}
\put(1751,-95){\makebox(0,0)[lb]{\smash{{\SetFigFont{9}{10.8}{\rmdefault}{\mddefault}{\updefault}{\color[rgb]{0,0,0}$b$}%
}}}}
\put(508,-754){\makebox(0,0)[lb]{\smash{{\SetFigFont{9}{10.8}{\rmdefault}{\mddefault}{\updefault}{\color[rgb]{0,0,0}$x$}%
}}}}
\put(1294,-926){\makebox(0,0)[lb]{\smash{{\SetFigFont{9}{10.8}{\rmdefault}{\mddefault}{\updefault}{\color[rgb]{0,0,0}$z$}%
}}}}
\put(648,364){\makebox(0,0)[lb]{\smash{{\SetFigFont{9}{10.8}{\rmdefault}{\mddefault}{\updefault}{\color[rgb]{0,0,0}$y$}%
}}}}
\put(1308,-166){\makebox(0,0)[lb]{\smash{{\SetFigFont{9}{10.8}{\rmdefault}{\mddefault}{\updefault}{\color[rgb]{0,0,0}$s$}%
}}}}
\put(1507,-168){\makebox(0,0)[lb]{\smash{{\SetFigFont{9}{10.8}{\rmdefault}{\mddefault}{\updefault}{\color[rgb]{0,0,0}$r$}%
}}}}
\put(2137,412){\makebox(0,0)[lb]{\smash{{\SetFigFont{9}{10.8}{\rmdefault}{\mddefault}{\updefault}{\color[rgb]{0,0,0}$q$}%
}}}}
\put(2407,406){\makebox(0,0)[lb]{\smash{{\SetFigFont{9}{10.8}{\rmdefault}{\mddefault}{\updefault}{\color[rgb]{0,0,0}$p$}%
}}}}
\put(2924,402){\makebox(0,0)[lb]{\smash{{\SetFigFont{9}{10.8}{\rmdefault}{\mddefault}{\updefault}{\color[rgb]{0,0,0}$a$}%
}}}}
\put(3359,400){\makebox(0,0)[lb]{\smash{{\SetFigFont{9}{10.8}{\rmdefault}{\mddefault}{\updefault}{\color[rgb]{0,0,0}$r$}%
}}}}
\put(3570,399){\makebox(0,0)[lb]{\smash{{\SetFigFont{9}{10.8}{\rmdefault}{\mddefault}{\updefault}{\color[rgb]{0,0,0}$s$}%
}}}}
\put(4148,572){\makebox(0,0)[lb]{\smash{{\SetFigFont{9}{10.8}{\rmdefault}{\mddefault}{\updefault}{\color[rgb]{0,0,0}$q$}%
}}}}
\put(4387,568){\makebox(0,0)[lb]{\smash{{\SetFigFont{9}{10.8}{\rmdefault}{\mddefault}{\updefault}{\color[rgb]{0,0,0}$p$}%
}}}}
\put(5603,558){\makebox(0,0)[lb]{\smash{{\SetFigFont{9}{10.8}{\rmdefault}{\mddefault}{\updefault}{\color[rgb]{0,0,0}$r$}%
}}}}
\put(5948,559){\makebox(0,0)[lb]{\smash{{\SetFigFont{9}{10.8}{\rmdefault}{\mddefault}{\updefault}{\color[rgb]{0,0,0}$s$}%
}}}}
\put(4896,-172){\makebox(0,0)[lb]{\smash{{\SetFigFont{9}{10.8}{\rmdefault}{\mddefault}{\updefault}{\color[rgb]{0,0,0}$a$}%
}}}}
\end{picture}%
\end{center}
\caption{Three cases of the projective butterfly theorem.}
\label{fig:CR}
\end{figure}
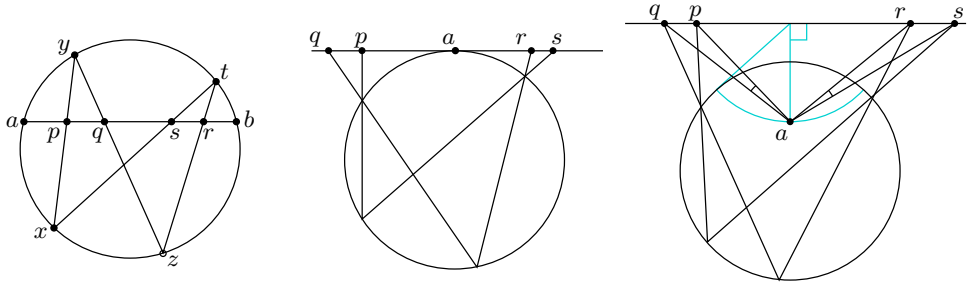

Theorem \ref{thm:CR} will be proved in the next section.

\begin{proof}[First proof of Theorem \ref{thm:Porism}]
If for some $x \in C$ the chain of chords closes, then by the first half of Theorem \ref{thm:CR} we have \eqref{eqn:CRCond}, respectively \eqref{eqn:Tangent}, or \eqref{eqn:ImagCR}. By the second half of the same theorem, the chain then closes for any $x \in C$.
\end{proof}

\subsection{Proving the projective butterfly theorem}
We need to extend the notion of the cross-ratio to concurrent lines and to points on a circle. This is classical and can be found for example in \cite[Chapter 16]{BerII}, \cite[Chapter VI]{Aud03}, or \cite[Appendix 4]{PT01}.

\begin{definition}
The \emph{cross-ratio} of four lines $\ell_1$, $\ell_2$, $\ell_3$, $\ell_4$ intersecting at a point $x$ is
\[
\CR(\ell_1, \ell_2; \ell_3, \ell_4) := \frac{\sin(\alpha-\gamma)}{\sin(\beta-\gamma)} : \frac{\sin(\alpha-\delta)}{\sin(\beta-\delta)}
\]
where $\alpha$, $\beta$, $\gamma$, $\delta$ are the angles from some line $\ell_0$ to $\ell_1$, $\ell_2$, $\ell_3$, $\ell_4$.
\end{definition}
Replacing $\alpha$ by $\pi+\alpha$ (taking the other oriented angle from $\ell_0$ to $\ell_1$) doesn't change the value of the cross-ratio.

\begin{prp}
\label{prp:CRLinesPoints}
Let $\ell$ be a line that doesn't go through $x$ and intersects the lines $\ell_1$, $\ell_2$, $\ell_3$, $\ell_4$ in the points $a$, $b$, $c$, $d$. Then we have
\[
\CR(\ell_1, \ell_2; \ell_3, \ell_4) = \CR(a, b; c, d).
\]
\end{prp}
\begin{proof}
The sine law implies
\[
\frac{a-c}{b-c} : \frac{a-d}{b-d} = \frac{\sin\angle axc}{\sin\angle bxc} : \frac{\sin\angle axd}{\sin\angle bxd}
\]
with signed lengths and angle values.
\end{proof}

Along the way we have proved that central projections (and hence all projective transformations) preserve the cross-ratio of collinear points, see Figure \ref{fig:CRPreserve}. This implies in turn that projective transformations preserve the cross-ratio of concurrent lines.

\begin{figure}[ht]
\begin{center}
\begin{picture}(0,0)%
\includegraphics{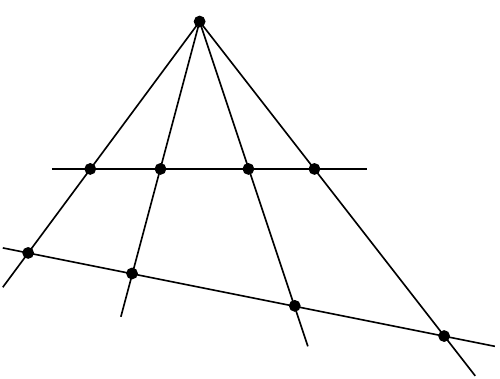}%
\end{picture}%
\setlength{\unitlength}{4144sp}%
\begingroup\makeatletter\ifx\SetFigFont\undefined%
\gdef\SetFigFont#1#2#3#4#5{%
  \reset@font\fontsize{#1}{#2pt}%
  \fontfamily{#3}\fontseries{#4}\fontshape{#5}%
  \selectfont}%
\fi\endgroup%
\begin{picture}(2274,1730)(-11,-793)
\put(1430,202){\makebox(0,0)[lb]{\smash{{\SetFigFont{9}{10.8}{\rmdefault}{\mddefault}{\updefault}{\color[rgb]{0,0,0}$b$}%
}}}}
\put( 82,-392){\makebox(0,0)[lb]{\smash{{\SetFigFont{9}{10.8}{\rmdefault}{\mddefault}{\updefault}{\color[rgb]{0,0,0}$a'$}%
}}}}
\put(586,-505){\makebox(0,0)[lb]{\smash{{\SetFigFont{9}{10.8}{\rmdefault}{\mddefault}{\updefault}{\color[rgb]{0,0,0}$c'$}%
}}}}
\put(1888,-724){\makebox(0,0)[lb]{\smash{{\SetFigFont{9}{10.8}{\rmdefault}{\mddefault}{\updefault}{\color[rgb]{0,0,0}$b'$}%
}}}}
\put(754,814){\makebox(0,0)[lb]{\smash{{\SetFigFont{9}{10.8}{\rmdefault}{\mddefault}{\updefault}{\color[rgb]{0,0,0}$x$}%
}}}}
\put(294,200){\makebox(0,0)[lb]{\smash{{\SetFigFont{9}{10.8}{\rmdefault}{\mddefault}{\updefault}{\color[rgb]{0,0,0}$a$}%
}}}}
\put(605,192){\makebox(0,0)[lb]{\smash{{\SetFigFont{9}{10.8}{\rmdefault}{\mddefault}{\updefault}{\color[rgb]{0,0,0}$c$}%
}}}}
\put(1139,192){\makebox(0,0)[lb]{\smash{{\SetFigFont{9}{10.8}{\rmdefault}{\mddefault}{\updefault}{\color[rgb]{0,0,0}$d$}%
}}}}
\put(1150,-609){\makebox(0,0)[lb]{\smash{{\SetFigFont{9}{10.8}{\rmdefault}{\mddefault}{\updefault}{\color[rgb]{0,0,0}$d'$}%
}}}}
\end{picture}%
\end{center}
\caption{$\frac{|ac|}{|bc|} : \frac{|ad|}{|bd|} = \frac{\sin\angle axc}{\sin\angle bxc} : \frac{\sin\angle axd}{\sin\angle bxd} = \frac{|a'c'|}{|b'c'|} : \frac{|a'd'|}{|b'd'|}$.}
\label{fig:CRPreserve}
\end{figure}
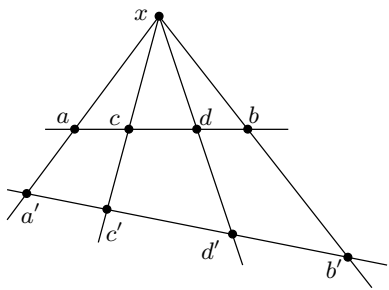

\begin{definition}
The cross-ratio of four points on a circle $C$ is defined as
\[
\CR(a, b; c, d) := \CR(xa, xb; xc, xd)
\]
where $x$ is an arbitrary point on $C$.
\end{definition}

\begin{figure}[ht]
\begin{center}
\begin{picture}(0,0)%
\includegraphics{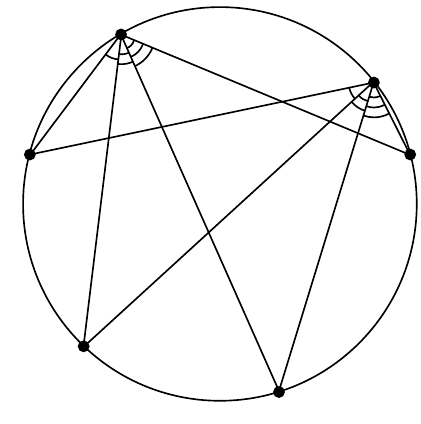}%
\end{picture}%
\setlength{\unitlength}{4144sp}%
\begingroup\makeatletter\ifx\SetFigFont\undefined%
\gdef\SetFigFont#1#2#3#4#5{%
  \reset@font\fontsize{#1}{#2pt}%
  \fontfamily{#3}\fontseries{#4}\fontshape{#5}%
  \selectfont}%
\fi\endgroup%
\begin{picture}(1920,1931)(346,-1515)
\put(2251,-331){\makebox(0,0)[lb]{\smash{{\SetFigFont{9}{10.8}{\rmdefault}{\mddefault}{\updefault}{\color[rgb]{0,0,0}$b$}%
}}}}
\put(361,-331){\makebox(0,0)[lb]{\smash{{\SetFigFont{9}{10.8}{\rmdefault}{\mddefault}{\updefault}{\color[rgb]{0,0,0}$a$}%
}}}}
\put(586,-1231){\makebox(0,0)[lb]{\smash{{\SetFigFont{9}{10.8}{\rmdefault}{\mddefault}{\updefault}{\color[rgb]{0,0,0}$c$}%
}}}}
\put(1642,-1460){\makebox(0,0)[lb]{\smash{{\SetFigFont{9}{10.8}{\rmdefault}{\mddefault}{\updefault}{\color[rgb]{0,0,0}$d$}%
}}}}
\put(777,293){\makebox(0,0)[lb]{\smash{{\SetFigFont{9}{10.8}{\rmdefault}{\mddefault}{\updefault}{\color[rgb]{0,0,0}$x$}%
}}}}
\put(2087, 62){\makebox(0,0)[lb]{\smash{{\SetFigFont{9}{10.8}{\rmdefault}{\mddefault}{\updefault}{\color[rgb]{0,0,0}$x'$}%
}}}}
\end{picture}%
\end{center}
\caption{Cross-ratio of points on a circle is well-defined.}
\label{fig:Angles}
\end{figure}
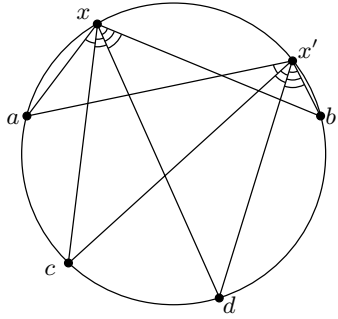

The right-hand side is independent of $x$ because of Figure \ref{fig:Angles}.
Now everything is ready to prove Theorem \ref{thm:CR}.

\begin{proof}[Proof of Theorem \ref{thm:CR}]
In the case of $\ell$ intersecting $C$, apply Proposition \ref{prp:CRLinesPoints} and the definition of the cross-ratio on the circle twice:
\begin{multline*}
\CR(a, b; p, q) = \CR(ya, yb; yp, yq) = \CR(a, b; x, z)\\
= \CR(ta, tb; tx, tz) = \CR(a, b; s, r).
\end{multline*}
For the inverse implication, use $\CR(a, b; s, r) = \CR(a, b; s, r') \Rightarrow r = r'$.

Incidencies of points, lines, and circles can be expressed by algebraic equations. Therefore, if \eqref{eqn:CRCond} holds in the case when the intersection points $a$ and $b$ are real, it holds in the case of complex intersection points, too. Equation \eqref{eqn:ImagCR} is a reformulation of \eqref{eqn:CRCond} in the case of complex $a$ and $b$.

We invite the reader to prove \eqref{eqn:Tangent} using similar triangles.
Equation \eqref{eqn:Tangent} is not (at least not directly) a limit case of \eqref{eqn:CRCond}, as $\CR(a, a; p, q) = 1$.
\end{proof}

Due to the projective invariance of the cross-ratio, \eqref{eqn:CRCond} holds if $C$ is an arbitrary ellipse, parabola, or hyperbola.

\subsection{More on butterflies}
The projective butterfly theorem immediately implies the classical butterfly theorem as well as its generalization due to Klamkin \cite{Kla65}, see Fig. \ref{fig:Butterfly}. This is the proof number ten (out of twenty) on~\cite{BogKnot}. A similar proof is given in \cite[Chapter~30]{Pra91}.

\begin{figure}[ht]
\begin{center}
\begin{picture}(0,0)%
\includegraphics{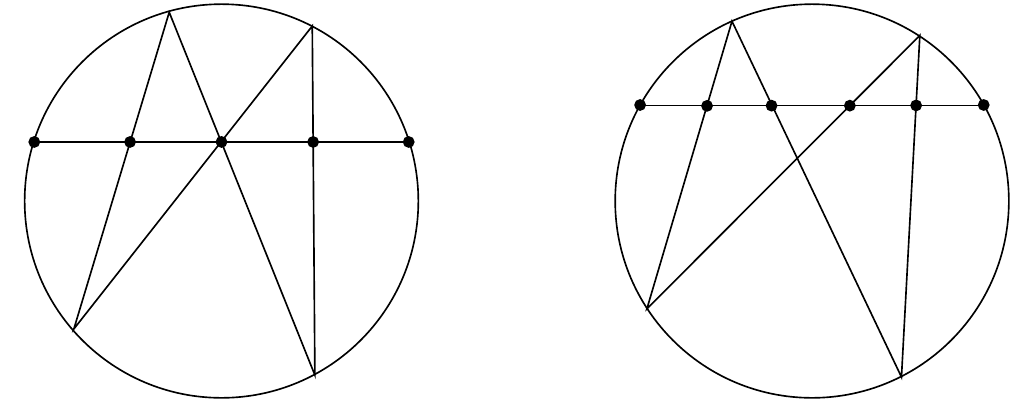}%
\end{picture}%
\setlength{\unitlength}{4144sp}%
\begingroup\makeatletter\ifx\SetFigFont\undefined%
\gdef\SetFigFont#1#2#3#4#5{%
  \reset@font\fontsize{#1}{#2pt}%
  \fontfamily{#3}\fontseries{#4}\fontshape{#5}%
  \selectfont}%
\fi\endgroup%
\begin{picture}(4620,1814)(-111,-968)
\put(-96,154){\makebox(0,0)[lb]{\smash{{\SetFigFont{9}{10.8}{\rmdefault}{\mddefault}{\updefault}{\color[rgb]{0,0,0}$a$}%
}}}}
\put(1801,141){\makebox(0,0)[lb]{\smash{{\SetFigFont{9}{10.8}{\rmdefault}{\mddefault}{\updefault}{\color[rgb]{0,0,0}$b$}%
}}}}
\put(484, 64){\makebox(0,0)[lb]{\smash{{\SetFigFont{9}{10.8}{\rmdefault}{\mddefault}{\updefault}{\color[rgb]{0,0,0}$p$}%
}}}}
\put(827, 22){\makebox(0,0)[lb]{\smash{{\SetFigFont{9}{10.8}{\rmdefault}{\mddefault}{\updefault}{\color[rgb]{0,0,0}$q$}%
}}}}
\put(1345, 85){\makebox(0,0)[lb]{\smash{{\SetFigFont{9}{10.8}{\rmdefault}{\mddefault}{\updefault}{\color[rgb]{0,0,0}$r$}%
}}}}
\put(2653,318){\makebox(0,0)[lb]{\smash{{\SetFigFont{9}{10.8}{\rmdefault}{\mddefault}{\updefault}{\color[rgb]{0,0,0}$a$}%
}}}}
\put(4428,321){\makebox(0,0)[lb]{\smash{{\SetFigFont{9}{10.8}{\rmdefault}{\mddefault}{\updefault}{\color[rgb]{0,0,0}$b$}%
}}}}
\put(3124,264){\makebox(0,0)[lb]{\smash{{\SetFigFont{9}{10.8}{\rmdefault}{\mddefault}{\updefault}{\color[rgb]{0,0,0}$p$}%
}}}}
\put(4085,250){\makebox(0,0)[lb]{\smash{{\SetFigFont{9}{10.8}{\rmdefault}{\mddefault}{\updefault}{\color[rgb]{0,0,0}$r$}%
}}}}
\put(3754,253){\makebox(0,0)[lb]{\smash{{\SetFigFont{9}{10.8}{\rmdefault}{\mddefault}{\updefault}{\color[rgb]{0,0,0}$s$}%
}}}}
\put(3316,259){\makebox(0,0)[lb]{\smash{{\SetFigFont{9}{10.8}{\rmdefault}{\mddefault}{\updefault}{\color[rgb]{0,0,0}$q$}%
}}}}
\end{picture}%
\end{center}
\caption{Left: $|aq|=|bq| \Rightarrow |pq|=|rq|$. Right: $|aq|=|bs| \Rightarrow |pq|=|rs|$.}
\label{fig:Butterfly}
\end{figure}
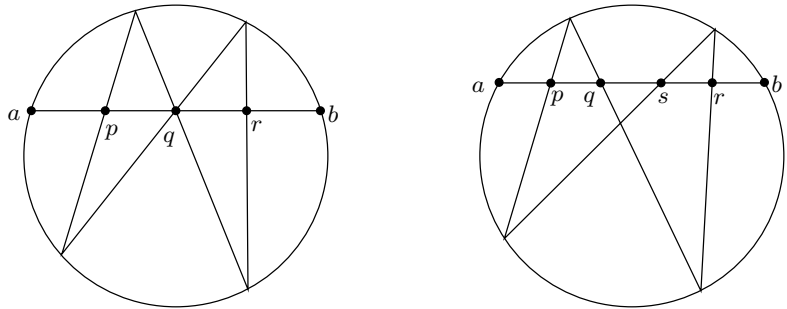

The case \eqref{eqn:CRCond} of the projective butterfly theorem appears in \cite{Brady}. Coxeter gives a special case as Exercise 2 in \cite[Chapter 8.3]{CoxProj}.

\section{Hyperbolic isometries and M\"obius transformations}
\label{sec:HypMoeb}
\subsection{The Cayley-Klein model of the hyperbolic plane}
In the \emph{Cayley-Klein model}, the points of the hyperbolic plane are the points inside a circle $C$, and the lines are the open chords of $C$. The distance between two points $p$ and $q$ is defined by the formula
\[
\dist(p,q) := \frac12 |\log\CR(a, b; p, q)|
\]
where $a$ and $b$ are the endpoints of the chord through $p$ and $q$.
% (The reader may check that for any $r$ on the chord $ab$ between $p$ and $q$ we have $\dist(p,r) + \dist(r,q) = \dist(p,q)$.)

Isometries of the hyperbolic plane in the Cayley-Klein model are those projective transformations that map the interior of $C$ to itself. There are plenty of them; for example, two triangles with equal corresponding sides are congruent, that is, they can be mapped one to the other by a hyperbolic isometry.

Theorem \ref{thm:CR} implies the following.

\begin{cor}
\label{cor:HypEqual}
The segments $pq$ and $rs$ on Figure \ref{fig:EqualLength}, left, have equal length in the Cayley-Klein model of the hyperbolic plane.
\end{cor}

\begin{figure}[ht]
\begin{center}
\begin{picture}(0,0)%
\includegraphics{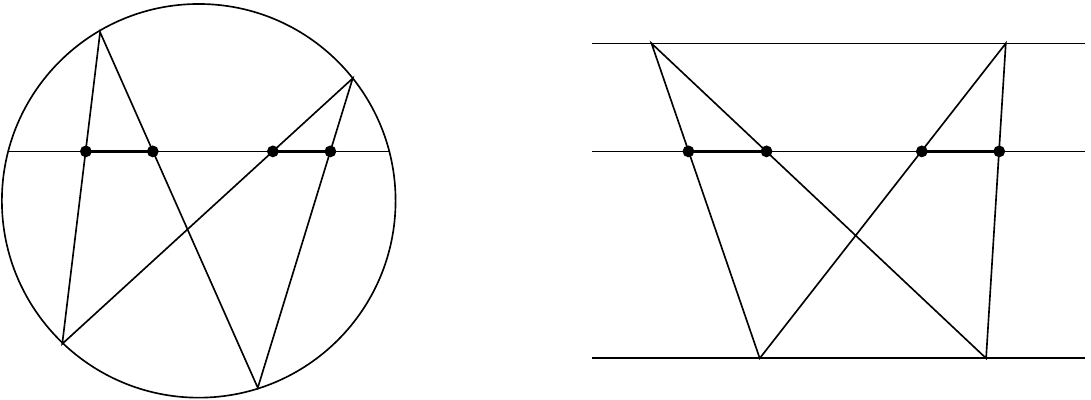}%
\end{picture}%
\setlength{\unitlength}{4144sp}%
\begingroup\makeatletter\ifx\SetFigFont\undefined%
\gdef\SetFigFont#1#2#3#4#5{%
  \reset@font\fontsize{#1}{#2pt}%
  \fontfamily{#3}\fontseries{#4}\fontshape{#5}%
  \selectfont}%
\fi\endgroup%
\begin{picture}(4970,1814)(443,-1418)
\put(676,-421){\makebox(0,0)[lb]{\smash{{\SetFigFont{9}{10.8}{\rmdefault}{\mddefault}{\updefault}{\color[rgb]{0,0,0}$p$}%
}}}}
\put(1036,-421){\makebox(0,0)[lb]{\smash{{\SetFigFont{9}{10.8}{\rmdefault}{\mddefault}{\updefault}{\color[rgb]{0,0,0}$q$}%
}}}}
\put(1666,-421){\makebox(0,0)[lb]{\smash{{\SetFigFont{9}{10.8}{\rmdefault}{\mddefault}{\updefault}{\color[rgb]{0,0,0}$s$}%
}}}}
\put(1936,-421){\makebox(0,0)[lb]{\smash{{\SetFigFont{9}{10.8}{\rmdefault}{\mddefault}{\updefault}{\color[rgb]{0,0,0}$r$}%
}}}}
\put(3478,-421){\makebox(0,0)[lb]{\smash{{\SetFigFont{9}{10.8}{\rmdefault}{\mddefault}{\updefault}{\color[rgb]{0,0,0}$p$}%
}}}}
\put(3892,-421){\makebox(0,0)[lb]{\smash{{\SetFigFont{9}{10.8}{\rmdefault}{\mddefault}{\updefault}{\color[rgb]{0,0,0}$q$}%
}}}}
\put(4609,-421){\makebox(0,0)[lb]{\smash{{\SetFigFont{9}{10.8}{\rmdefault}{\mddefault}{\updefault}{\color[rgb]{0,0,0}$s$}%
}}}}
\put(5028,-421){\makebox(0,0)[lb]{\smash{{\SetFigFont{9}{10.8}{\rmdefault}{\mddefault}{\updefault}{\color[rgb]{0,0,0}$r$}%
}}}}
\end{picture}%
\end{center}
\caption{Addition of segments in hyperbolic and euclidean geometry.}
\label{fig:EqualLength}
\end{figure}
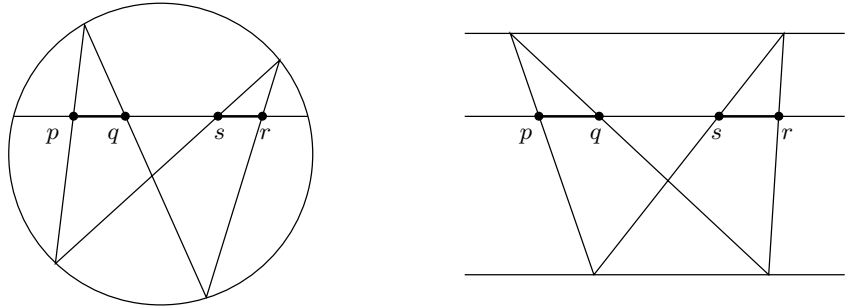

This gives a ruler construction for collinear segments of equal length with arbitrary initial points. 
Compare this with the construction of equal segments in the euclidean plane, given two lines parallel to $\ell$, in Figure \ref{fig:EqualLength}, right.

Every point $p$ outside $C$ corresponds by polarity to a hyperbolic line $p^\circ$. If one or both of $p$ and $q$ lie outside $C$, then the cross-ratio $\CR(a, b; p, q)$ is related to the distance between a point and a line, the angle between two intersecting lines, or the distance between two disjoint lines, see Figure \ref{fig:PointsLines}.

\begin{figure}[ht]
\begin{center}
\begin{picture}(0,0)%
\includegraphics{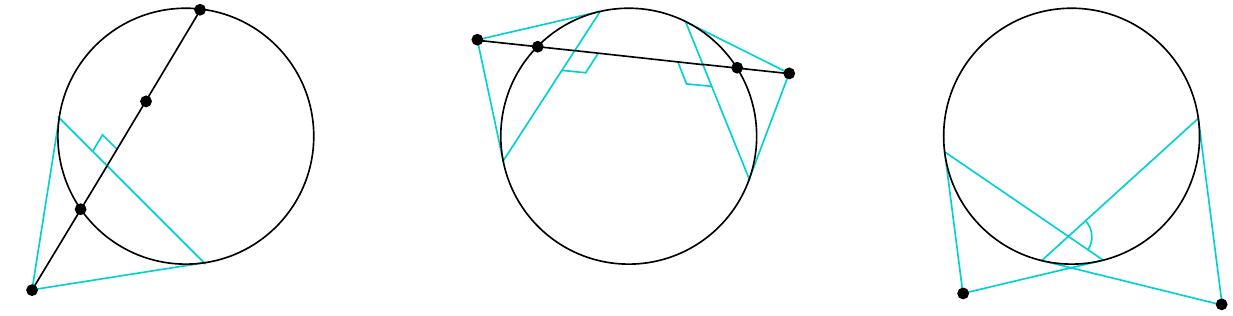}%
\end{picture}%
\setlength{\unitlength}{4144sp}%
\begingroup\makeatletter\ifx\SetFigFont\undefined%
\gdef\SetFigFont#1#2#3#4#5{%
  \reset@font\fontsize{#1}{#2pt}%
  \fontfamily{#3}\fontseries{#4}\fontshape{#5}%
  \selectfont}%
\fi\endgroup%
\begin{picture}(5624,1430)(-173,-658)
\put(-158,-520){\makebox(0,0)[lb]{\smash{{\SetFigFont{9}{10.8}{\rmdefault}{\mddefault}{\updefault}{\color[rgb]{0,0,0}$p$}%
}}}}
\put(519,211){\makebox(0,0)[lb]{\smash{{\SetFigFont{9}{10.8}{\rmdefault}{\mddefault}{\updefault}{\color[rgb]{0,0,0}$q$}%
}}}}
\put(715,595){\makebox(0,0)[lb]{\smash{{\SetFigFont{9}{10.8}{\rmdefault}{\mddefault}{\updefault}{\color[rgb]{0,0,0}$b$}%
}}}}
\put(238,-212){\makebox(0,0)[lb]{\smash{{\SetFigFont{9}{10.8}{\rmdefault}{\mddefault}{\updefault}{\color[rgb]{0,0,0}$a$}%
}}}}
\put(1872,630){\makebox(0,0)[lb]{\smash{{\SetFigFont{9}{10.8}{\rmdefault}{\mddefault}{\updefault}{\color[rgb]{0,0,0}$p$}%
}}}}
\put(2265,442){\makebox(0,0)[lb]{\smash{{\SetFigFont{9}{10.8}{\rmdefault}{\mddefault}{\updefault}{\color[rgb]{0,0,0}$a$}%
}}}}
\put(3443,486){\makebox(0,0)[lb]{\smash{{\SetFigFont{9}{10.8}{\rmdefault}{\mddefault}{\updefault}{\color[rgb]{0,0,0}$q$}%
}}}}
\put(4097,-541){\makebox(0,0)[lb]{\smash{{\SetFigFont{9}{10.8}{\rmdefault}{\mddefault}{\updefault}{\color[rgb]{0,0,0}$p$}%
}}}}
\put(3109,334){\makebox(0,0)[lb]{\smash{{\SetFigFont{9}{10.8}{\rmdefault}{\mddefault}{\updefault}{\color[rgb]{0,0,0}$b$}%
}}}}
\put(5436,-598){\makebox(0,0)[lb]{\smash{{\SetFigFont{9}{10.8}{\rmdefault}{\mddefault}{\updefault}{\color[rgb]{0,0,0}$q$}%
}}}}
\put(2982,131){\makebox(0,0)[lb]{\smash{{\SetFigFont{9}{10.8}{\rmdefault}{\mddefault}{\updefault}{\color[rgb]{0,.82,.82}$q^\circ$}%
}}}}
\put(2231,149){\makebox(0,0)[lb]{\smash{{\SetFigFont{9}{10.8}{\rmdefault}{\mddefault}{\updefault}{\color[rgb]{0,.82,.82}$p^\circ$}%
}}}}
\put(4938, 60){\makebox(0,0)[lb]{\smash{{\SetFigFont{9}{10.8}{\rmdefault}{\mddefault}{\updefault}{\color[rgb]{0,.82,.82}$q^\circ$}%
}}}}
\put(4311, -8){\makebox(0,0)[lb]{\smash{{\SetFigFont{9}{10.8}{\rmdefault}{\mddefault}{\updefault}{\color[rgb]{0,.82,.82}$p^\circ$}%
}}}}
\put(615,-235){\makebox(0,0)[lb]{\smash{{\SetFigFont{9}{10.8}{\rmdefault}{\mddefault}{\updefault}{\color[rgb]{0,.82,.82}$p^\circ$}%
}}}}
\end{picture}%
\end{center}
\caption{All distances and angles in the hyperbolic plane are related to the cross-ratio.}
\label{fig:PointsLines}
\end{figure}
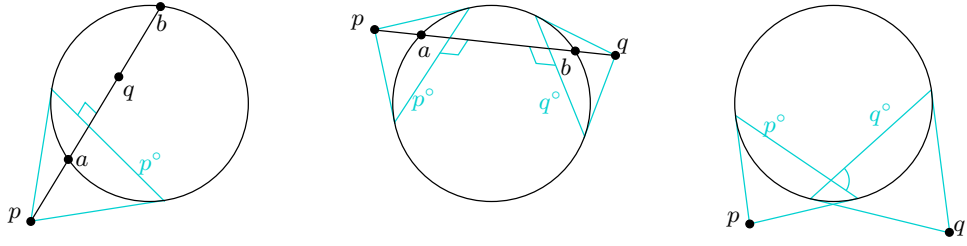

% Corollary \ref{cor:HypEqual} also explains the appearence in \cite{Kocik13} of the formula for relativistic addition of velocities: in a coordinate system on $\ell$, in which the points $p, q, s$ have coordinates $0, u, v$ respectively, the point $r$ has the coordinate $\frac{u+v}{1+uv}$. Note that this formula is the addition formula for $\tanh$ (and $\tanh$ corresponds to the velocity). On the other hand, Corollary \ref{cor:HypEqual} tells us that $r = \tanh(\phi+\psi)$ as soon as $p = \tanh 0$, $q = \tanh\phi$ and $s = \tanh\psi$.

\subsection{Types of hyperbolic isometries and M\"obius transformations}
\begin{definition}
A \emph{M\"obius transformation} of a circle $C$ is a restriction of a projective transformation that leaves $C$ invariant.
\end{definition}
In other words, a M\"obius transformation is an extension of a hyperbolic isometry by continuity to the boundary of the Cayley-Klein model.

The action of the M\"obius group on a circle is equivalent to the action of the linear fractional group on a projective line, if the circle is identified with the line by a stereographic projection. This implies the following.
\begin{prp}
M\"obius transformations of a circle act freely and transitively on triples of points.
\end{prp}

In particular, a non-identical M\"obius transformation cannot have more than two fixed points. An orientation-preserving M\"obius transformation $f$, together with the corresponding orientation-preserving hyperbolic isometry $\tilde{f}$, belongs to one of the following three types:
\begin{itemize}
\item \emph{elliptic}, if $f$ has no fixed points;
\item \emph{parabolic}, if $f$ has exactly one fixed point;
\item \emph{hyperbolic}, if $f$ has exactly two fixed points.
\end{itemize}

Every hyperbolic isometry of elliptic type has one fixed point inside $C$, and is therefore a rotation around this point. Rotation by $180^\circ$ is the point reflection.

\begin{figure}[ht]
\begin{center}
\begin{picture}(0,0)%
\includegraphics{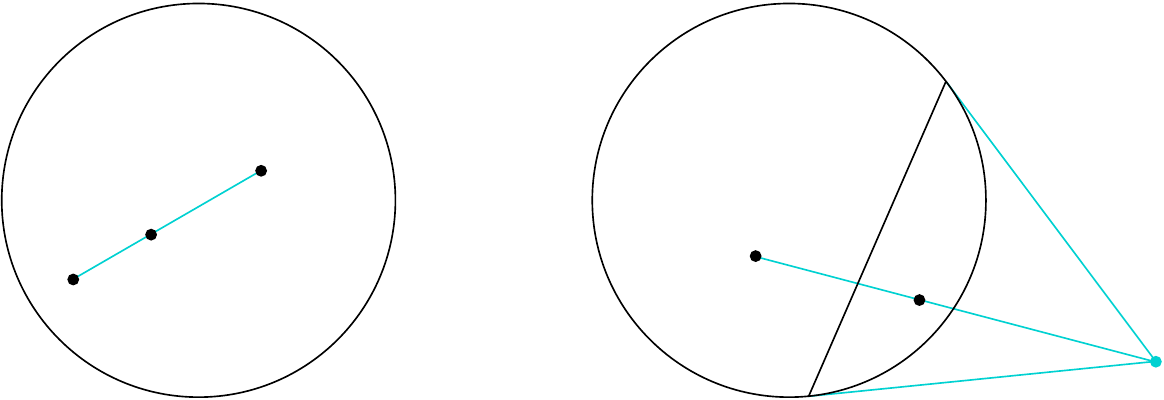}%
\end{picture}%
\setlength{\unitlength}{4144sp}%
\begingroup\makeatletter\ifx\SetFigFont\undefined%
\gdef\SetFigFont#1#2#3#4#5{%
  \reset@font\fontsize{#1}{#2pt}%
  \fontfamily{#3}\fontseries{#4}\fontshape{#5}%
  \selectfont}%
\fi\endgroup%
\begin{picture}(5317,1816)(-7,-970)
\put(692,-352){\makebox(0,0)[lb]{\smash{{\SetFigFont{9}{10.8}{\rmdefault}{\mddefault}{\updefault}{\color[rgb]{0,0,0}$p$}%
}}}}
\put(340,-566){\makebox(0,0)[lb]{\smash{{\SetFigFont{9}{10.8}{\rmdefault}{\mddefault}{\updefault}{\color[rgb]{0,0,0}$x$}%
}}}}
\put(1205,-58){\makebox(0,0)[lb]{\smash{{\SetFigFont{9}{10.8}{\rmdefault}{\mddefault}{\updefault}{\color[rgb]{0,0,0}$R_p(x)$}%
}}}}
\put(3973,-34){\makebox(0,0)[lb]{\smash{{\SetFigFont{9}{10.8}{\rmdefault}{\mddefault}{\updefault}{\color[rgb]{0,0,0}$\ell$}%
}}}}
\put(5295,-763){\makebox(0,0)[lb]{\smash{{\SetFigFont{9}{10.8}{\rmdefault}{\mddefault}{\updefault}{\color[rgb]{0,.82,.82}$\ell^\circ$}%
}}}}
\put(3300,-493){\makebox(0,0)[lb]{\smash{{\SetFigFont{9}{10.8}{\rmdefault}{\mddefault}{\updefault}{\color[rgb]{0,0,0}$S_\ell(x)$}%
}}}}
\put(4219,-484){\makebox(0,0)[lb]{\smash{{\SetFigFont{9}{10.8}{\rmdefault}{\mddefault}{\updefault}{\color[rgb]{0,0,0}$x$}%
}}}}
\end{picture}%
\end{center}
\caption{Point and line reflections in the Cayley-Klein model.}
\label{fig:Reflections}
\end{figure}
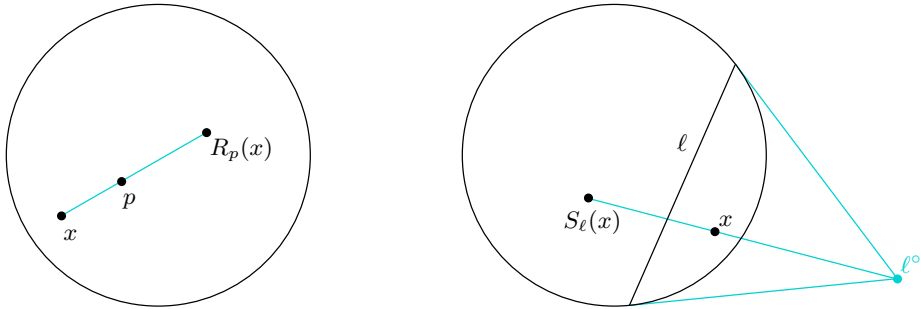

The most basic orientation-reversing hyperbolic isometry is the line reflection. The lines joining a point with its mirror image with respect to $\ell$ all go through the pole of~$\ell$.
% It is also useful to note that the composition of two line reflections is of elliptic, parabolic, or hyperbolic type depending on whether the lines intersect, are parallel, or are ultraparallel.

\subsection{The porism as composition of M\"obius involutions}
For a point $p$ not on a circle $C$, define a map
\begin{equation}
\label{eqn:Ip}
I_p \colon C \to C
\end{equation}
so that $I_p(x)$ is the other end of the chord from $x$ through $p$. Then Theorem \ref{thm:Porism} can be reformulated as follows.

\begin{theorem}
For a circle $C$ and a line $\ell$, let $p, q, r, s \in \ell \setminus C$. If there exists a point $x \in C \setminus \ell$ such that $I_s \circ I_r \circ I_q \circ I_p(x) = x$, then $I_s \circ I_r \circ I_q \circ I_p = \mathrm{id}$.
\end{theorem}
Interpretation of $I_p$ as a M\"obius transformation leads to an almost straightforward proof.
\begin{proof}
For $p$ inside $C$, the map $I_p$ extends by continuity the rotation by $\pi$ around $p$ in the Cayley-Klein model; for $p$ outside $C$, the map $I_p$ is an extension of a reflection in the polar line of $p$, see Figure \ref{fig:Reflections}. In any case, $I_p$ is a M\"obius transformation of $C$.

An intersection point of $\ell$ with $C$ is always a fixed point of the composition $I_s \circ I_r \circ I_q \circ I_p$.

If $|\ell \cap C| = 2$, then the existence of a third fixed point $x$ implies that $I_s \circ I_r \circ I_q \circ I_p$ is the identity.

Assume that $\ell$ is tangent to $C$ in a point $a$. Then $a$ is a unique fixed point of $I_q \circ I_p$, as well as of $I_s \circ I_r$. It follows that both maps are of parabolic type. The composition of two parabolic rotations with the same center is again a parabolic rotation or identity. Thus if $I_s \circ I_r \circ I_q \circ I_p$ has a fixed point $x$ different from $a$, then it must be the identity.

Finally, if $\ell$ is disjoint from $a$, then $I_q \circ I_p$ and $I_s \circ I_r$ are elliptics with the center $p^\circ \cap q^\circ = \ell^\circ = r^\circ \cap s^\circ$. Their composition is again an elliptic with the same center, or identity. If it has a fixed point $x \in C$, then it must be the identity, and the theorem is proved.
\end{proof}

This sheds new light on the projective butterfly theorem. Let us show that each of the conditions \eqref{eqn:CRCond}, \eqref{eqn:Tangent}, \eqref{eqn:ImagCR} is equivalent to
\begin{equation}
\label{eqn:EqualMaps}
I_q \circ I_p = I_r \circ I_s.
\end{equation}
In the first of the cases on Figure \ref{fig:CR}, the map $I_q \circ I_p$ is of hyperbolic type and corresponds to translation along the axis $pq$ by the distance $2\dist(p,q)$. In order for \eqref{eqn:EqualMaps} to hold, we thus must have $\dist(p,q) = \dist(r,s)$ (together with a correct order of points), and condition \eqref{eqn:CRCond} says exactly that.

In the second case, $I_q \circ I_p$ is of parabolic type, and the corresponding hyperbolic isometry maps a horocycle with center at $a$ to itself. The number $\frac{1}{a-p} - \frac{1}{a-q}$ is proportional to the translation length along the horocycle (the horocyclic coordinate can be measured by the stereographic projection from the point $a$), so that \eqref{eqn:Tangent} is again equivalent to \eqref{eqn:EqualMaps}.

Finally, in the third case the map $I_q \circ I_p$ is of elliptic type and corresponds to a rotation with center $\ell^\circ$. The angle $paq$ is half the rotation angle (to see this, use the Poincar\'e model); hence $\angle paq = \angle sar$ is equivalent to \eqref{eqn:EqualMaps}.

\subsection{Castillon's problem}
Castillon's problem is as follows.
\begin{quote}
Given a circle $C$ and $n$ points $p_1, \ldots, p_n$ not on $C$, inscribe in $C$ an $n$-gon whose sides (or their extensions) pass consecutively through $p_1, \ldots, p_n$.
\end{quote}

In other words, the problem consists of finding a fixed point of the map $I_{p_n} \circ \cdots \circ I_{p_1} \colon C \to C$. Berger \cite[Section 16.3.10.3]{BerII} describes how projective-geometric considerations lead to an actual construction procedure. The original 18\textsuperscript{th} century problem was posed for three points, that is for an inscribed triangle, \cite[Problem 29]{Doerrie}.

The fact that a M\"obius transformation with three fixed points is the identity implies the following.

\begin{theorem}
If the Castillon's problem has at least three solutions, then every starting point on the circle gives a solution.

If $p_1, \ldots, p_n$ are collinear, then for $n$ odd there is no solution, and for $n$ even the existence of one non-trivial solution implies that every starting point gives a solution.
\end{theorem}
For $n=3$ the problem has at most two solutions.
For $n=4$ infinitely many solutions are possible only if $p_1,p_2,p_3,p_4$ are collinear, since the axes, respectively centers of the maps in \eqref{eqn:EqualMaps}, must coincide.
The case of six points coinciding in the pairs $p_1 = p_4$, $p_2 = p_5$, $p_3 = p_6$ is related to Pascal's theorem.

Hyperbolic geometry allows us to produce configurations of points, for which the chain of chords closes independently of the starting point.
\begin{theorem}
Let $p_1, \ldots, p_n$ be points inside a circle such that $p_1p_2\ldots p_n$ is a right-angled polygon when the interior of the circle is viewed as the Cayley-Klein model of the hyperbolic plane. Then $I_{p_n} \circ \cdots \circ I_{p_1} = \mathrm{id}$.
\end{theorem}
\begin{proof}
In a manner similar to euclidean geometry, the reflection in a point $p$ can be represented as composition of reflections in two orthogonal lines through~$p$. Replace each $I_{p_k}$ by the composition of reflections in the adjacent sides of the $n$-gon. Then we have
\[
I_{p_n} \circ \cdots \circ I_{p_1} = (S_1 \circ S_n) \circ (S_n \circ S_{n-1}) \circ \cdots \circ (S_2 \circ S_1) = S_1 \circ S_1 = \mathrm{id}. \qedhere
\]
\end{proof}
See Fig. \ref{fig:Castillon5} for the construction of a right-angled pentagon and two examples of closed trajectories.
% One can also show that for any other configuration of $6$ points producing the identity the lengths $p_1p_2$, $p_3p_4$, $p_5p_6$ and the pairwise distances between these three lines must be side lengths of some right-angled hexagon.

\begin{figure}[ht]
\begin{center}
\includegraphics[width=.9\textwidth]{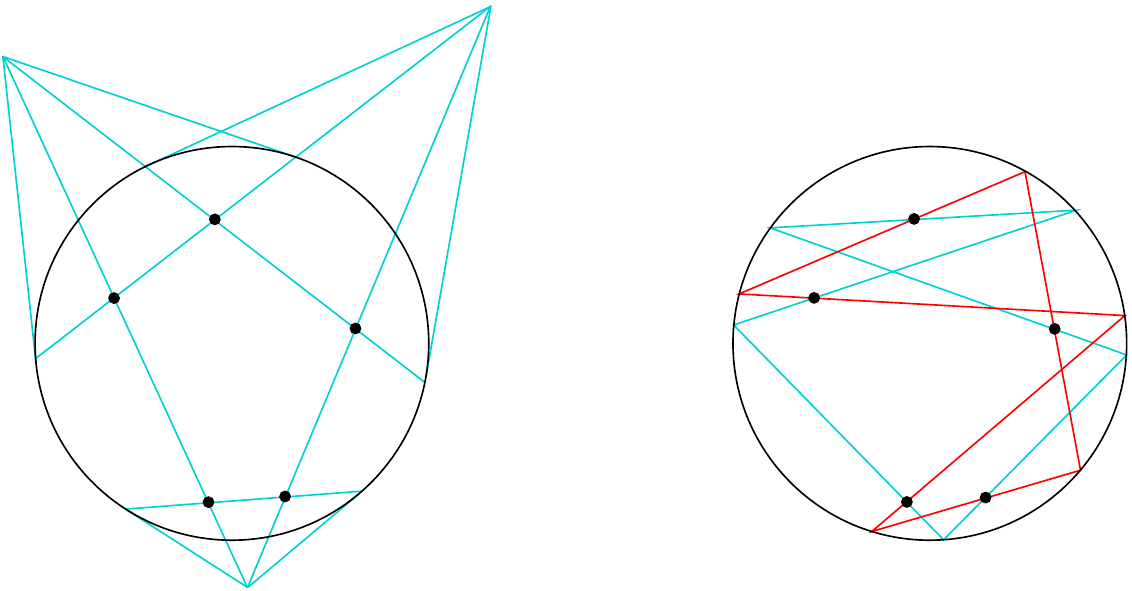}
\end{center}
\caption{Castillon's problem for vertices of a right-angled hyperbolic pentagon has infinitely many solutions.}
\label{fig:Castillon5}
\end{figure}

Castillon's problem can be posed for the sphere. The map $I_p \colon \Sph^2 \to \Sph^2$ is a M\"obius transformation. Since an orientation-preserving M\"obius transformation of $\Sph^2$ is determined by the images of three points, we have the following.

\begin{theorem}
If the Castillon's problem for the sphere and an even number of points has at least three solutions, then every starting point on the sphere gives a solution.
\end{theorem}

Figure \ref{fig:Dodeca} shows three configuration of points, for which the chain of chords always closes. The points are the vertices of a right-angled hyperbolic dodecahedron in the Cayley-Klein model.

\begin{figure}[ht]
\begin{center}
\begin{picture}(0,0)%
\includegraphics{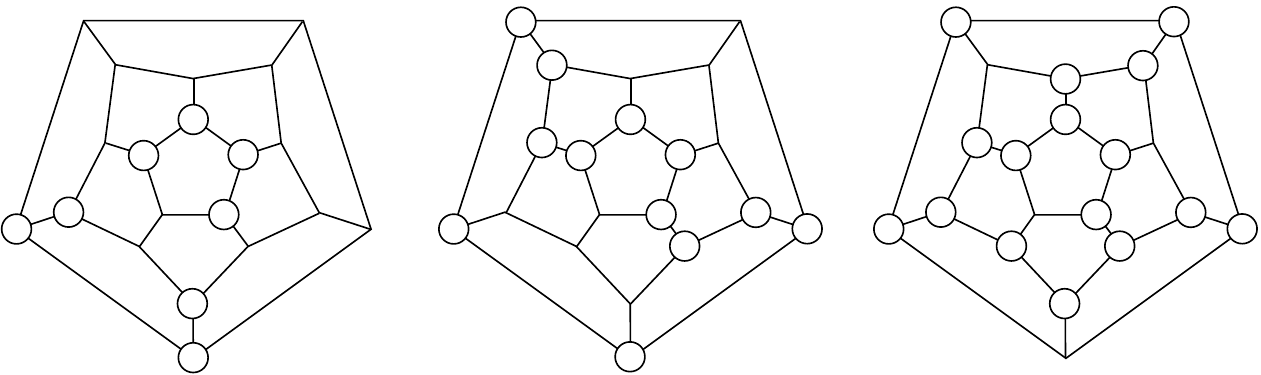}%
\end{picture}%
\setlength{\unitlength}{4144sp}%
\begingroup\makeatletter\ifx\SetFigFont\undefined%
\gdef\SetFigFont#1#2#3#4#5{%
  \reset@font\fontsize{#1}{#2pt}%
  \fontfamily{#3}\fontseries{#4}\fontshape{#5}%
  \selectfont}%
\fi\endgroup%
\begin{picture}(5756,1686)(-817,-541)
\put(3799,432){\makebox(0,0)[lb]{\smash{{\SetFigFont{6}{7.2}{\rmdefault}{\mddefault}{\updefault}{\color[rgb]{0,0,0}$1$}%
}}}}
\put(3424,168){\makebox(0,0)[lb]{\smash{{\SetFigFont{6}{7.2}{\rmdefault}{\mddefault}{\updefault}{\color[rgb]{0,0,0}$10$}%
}}}}
\put(3187, 89){\makebox(0,0)[lb]{\smash{{\SetFigFont{6}{7.2}{\rmdefault}{\mddefault}{\updefault}{\color[rgb]{0,0,0}$11$}%
}}}}
\put(3747, 14){\makebox(0,0)[lb]{\smash{{\SetFigFont{6}{7.2}{\rmdefault}{\mddefault}{\updefault}{\color[rgb]{0,0,0}$16$}%
}}}}
\put(4029,592){\makebox(0,0)[lb]{\smash{{\SetFigFont{6}{7.2}{\rmdefault}{\mddefault}{\updefault}{\color[rgb]{0,0,0}$2$}%
}}}}
\put(4256,428){\makebox(0,0)[lb]{\smash{{\SetFigFont{6}{7.2}{\rmdefault}{\mddefault}{\updefault}{\color[rgb]{0,0,0}$3$}%
}}}}
\put(4171,156){\makebox(0,0)[lb]{\smash{{\SetFigFont{6}{7.2}{\rmdefault}{\mddefault}{\updefault}{\color[rgb]{0,0,0}$4$}%
}}}}
\put(4278, 13){\makebox(0,0)[lb]{\smash{{\SetFigFont{6}{7.2}{\rmdefault}{\mddefault}{\updefault}{\color[rgb]{0,0,0}$5$}%
}}}}
\put(3996,-254){\makebox(0,0)[lb]{\smash{{\SetFigFont{6}{7.2}{\rmdefault}{\mddefault}{\updefault}{\color[rgb]{0,0,0}$15$}%
}}}}
\put(4384,841){\makebox(0,0)[lb]{\smash{{\SetFigFont{6}{7.2}{\rmdefault}{\mddefault}{\updefault}{\color[rgb]{0,0,0}$7$}%
}}}}
\put(4031,775){\makebox(0,0)[lb]{\smash{{\SetFigFont{6}{7.2}{\rmdefault}{\mddefault}{\updefault}{\color[rgb]{0,0,0}$8$}%
}}}}
\put(3621,488){\makebox(0,0)[lb]{\smash{{\SetFigFont{6}{7.2}{\rmdefault}{\mddefault}{\updefault}{\color[rgb]{0,0,0}$9$}%
}}}}
\put(3500,1030){\makebox(0,0)[lb]{\smash{{\SetFigFont{6}{7.2}{\rmdefault}{\mddefault}{\updefault}{\color[rgb]{0,0,0}$12$}%
}}}}
\put(4606,169){\makebox(0,0)[lb]{\smash{{\SetFigFont{6}{7.2}{\rmdefault}{\mddefault}{\updefault}{\color[rgb]{0,0,0}$6$}%
}}}}
\put(4493,1035){\makebox(0,0)[lb]{\smash{{\SetFigFont{6}{7.2}{\rmdefault}{\mddefault}{\updefault}{\color[rgb]{0,0,0}$13$}%
}}}}
\put(4804, 91){\makebox(0,0)[lb]{\smash{{\SetFigFont{6}{7.2}{\rmdefault}{\mddefault}{\updefault}{\color[rgb]{0,0,0}$14$}%
}}}}
\put(1810,432){\makebox(0,0)[lb]{\smash{{\SetFigFont{6}{7.2}{\rmdefault}{\mddefault}{\updefault}{\color[rgb]{0,0,0}$1$}%
}}}}
\put(2040,592){\makebox(0,0)[lb]{\smash{{\SetFigFont{6}{7.2}{\rmdefault}{\mddefault}{\updefault}{\color[rgb]{0,0,0}$2$}%
}}}}
\put(2267,428){\makebox(0,0)[lb]{\smash{{\SetFigFont{6}{7.2}{\rmdefault}{\mddefault}{\updefault}{\color[rgb]{0,0,0}$3$}%
}}}}
\put(2182,156){\makebox(0,0)[lb]{\smash{{\SetFigFont{6}{7.2}{\rmdefault}{\mddefault}{\updefault}{\color[rgb]{0,0,0}$4$}%
}}}}
\put(2289, 13){\makebox(0,0)[lb]{\smash{{\SetFigFont{6}{7.2}{\rmdefault}{\mddefault}{\updefault}{\color[rgb]{0,0,0}$5$}%
}}}}
\put(-189,432){\makebox(0,0)[lb]{\smash{{\SetFigFont{6}{7.2}{\rmdefault}{\mddefault}{\updefault}{\color[rgb]{0,0,0}$1$}%
}}}}
\put(268,428){\makebox(0,0)[lb]{\smash{{\SetFigFont{6}{7.2}{\rmdefault}{\mddefault}{\updefault}{\color[rgb]{0,0,0}$3$}%
}}}}
\put(183,156){\makebox(0,0)[lb]{\smash{{\SetFigFont{6}{7.2}{\rmdefault}{\mddefault}{\updefault}{\color[rgb]{0,0,0}$4$}%
}}}}
\put(2034,-491){\makebox(0,0)[lb]{\smash{{\SetFigFont{6}{7.2}{\rmdefault}{\mddefault}{\updefault}{\color[rgb]{0,0,0}$8$}%
}}}}
\put(2617,169){\makebox(0,0)[lb]{\smash{{\SetFigFont{6}{7.2}{\rmdefault}{\mddefault}{\updefault}{\color[rgb]{0,0,0}$6$}%
}}}}
\put(2848, 96){\makebox(0,0)[lb]{\smash{{\SetFigFont{6}{7.2}{\rmdefault}{\mddefault}{\updefault}{\color[rgb]{0,0,0}$7$}%
}}}}
\put(1231, 93){\makebox(0,0)[lb]{\smash{{\SetFigFont{6}{7.2}{\rmdefault}{\mddefault}{\updefault}{\color[rgb]{0,0,0}$9$}%
}}}}
\put(1509,1034){\makebox(0,0)[lb]{\smash{{\SetFigFont{6}{7.2}{\rmdefault}{\mddefault}{\updefault}{\color[rgb]{0,0,0}$10$}%
}}}}
\put(1654,835){\makebox(0,0)[lb]{\smash{{\SetFigFont{6}{7.2}{\rmdefault}{\mddefault}{\updefault}{\color[rgb]{0,0,0}$11$}%
}}}}
\put(1606,483){\makebox(0,0)[lb]{\smash{{\SetFigFont{6}{7.2}{\rmdefault}{\mddefault}{\updefault}{\color[rgb]{0,0,0}$12$}%
}}}}
\put(-770, 93){\makebox(0,0)[lb]{\smash{{\SetFigFont{6}{7.2}{\rmdefault}{\mddefault}{\updefault}{\color[rgb]{0,0,0}$7$}%
}}}}
\put( 37,-246){\makebox(0,0)[lb]{\smash{{\SetFigFont{6}{7.2}{\rmdefault}{\mddefault}{\updefault}{\color[rgb]{0,0,0}$5$}%
}}}}
\put( 45,-491){\makebox(0,0)[lb]{\smash{{\SetFigFont{6}{7.2}{\rmdefault}{\mddefault}{\updefault}{\color[rgb]{0,0,0}$6$}%
}}}}
\put(-534,170){\makebox(0,0)[lb]{\smash{{\SetFigFont{6}{7.2}{\rmdefault}{\mddefault}{\updefault}{\color[rgb]{0,0,0}$8$}%
}}}}
\put(-295,242){\makebox(0,0)[lb]{\smash{{\SetFigFont{6}{7.2}{\rmdefault}{\mddefault}{\updefault}{\color[rgb]{0,0,0}$a$}%
}}}}
\put( 34, -5){\makebox(0,0)[lb]{\smash{{\SetFigFont{6}{7.2}{\rmdefault}{\mddefault}{\updefault}{\color[rgb]{0,0,0}$c$}%
}}}}
\put( 41,592){\makebox(0,0)[lb]{\smash{{\SetFigFont{6}{7.2}{\rmdefault}{\mddefault}{\updefault}{\color[rgb]{0,0,0}$2$}%
}}}}
\put(255,662){\makebox(0,0)[lb]{\smash{{\SetFigFont{6}{7.2}{\rmdefault}{\mddefault}{\updefault}{\color[rgb]{0,0,0}$b$}%
}}}}
\put(-516,-318){\makebox(0,0)[lb]{\smash{{\SetFigFont{6}{7.2}{\rmdefault}{\mddefault}{\updefault}{\color[rgb]{0,0,0}$d$}%
}}}}
\end{picture}%
\end{center}
\caption{Configurations of points inside $\Sph^2$, for which the Castillon's problem has infinitely many solutions.}
\label{fig:Dodeca}
\end{figure}
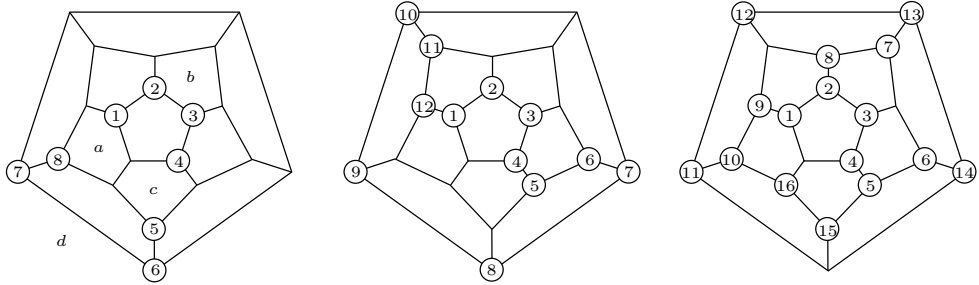

To see why these configurations work, represent the reflection in a vertex as a composition of reflections in three adjacent sides of the dodecahedron. By cancelling the common reflection, we get for example $I_2 \circ I_1 = S_b \circ S_a$ in the leftmost figure, so that
\[
I_8 \circ \cdots \circ I_1 = (S_a \circ S_d) \circ (S_d \circ S_c) \circ (S_c \circ S_b) \circ (S_b \circ S_a) = S_a \circ S_a = \mathrm{id}
\]
where $S_x$ denotes the reflection in the face $x$.

\def\cprime{$'$}

\end{document}